\documentclass[12pt]{article}

\usepackage{amssymb, amsmath, amsthm, graphicx}
\usepackage[left=1in,top=1in,right=1in]{geometry}
\usepackage[color=Orange!50!white, textwidth=22mm]{todonotes}

\usepackage{pgfplots}
\pgfplotsset{compat=1.15}
\usepackage{mathrsfs}
\usetikzlibrary{arrows}

\theoremstyle{plain}
      \newtheorem{theorem}{Theorem}
      \newtheorem{lemma}[theorem]{Lemma}

\theoremstyle{definition}

\theoremstyle{remark}

\long\def\comment#1{}

\title{How large part of a graph can be covered\\ by the neighborhoods of $k$ vertices? }
\author{J\'anos Pach\thanks{Alfr\'ed R\'enyi Institute of Mathematics, Budapest, Hungary. Email: \texttt{pach@renyi.hu}.}}

\begin{document}
\date{}

\maketitle

\begin{abstract}
Let $k\ge 2$ be fixed integer, $0<c<1$ a constant. Consider a graph $G$ with $n$ vertices and average degree $cn$. We answer a question of Simon Griffiths by showing that $G$ has $k$ vertices such that their neighborhoods together cover at least $\min(1-(1-c)^{k},\sqrt{c})n$ vertices. This result is essentially tight.
\end{abstract}

\section{Introduction}


Let $G$ be a fixed graph with vertex set $V$ and edge set $E$, where $n=|V|.$  For any vertex $u\in V,$ let $N(u)$ denote the \emph{neighborhood} of $u$ in $G$, that is $N(u)=\{x\in V : ux\in E\}$. The size of $N(u)$ is the \emph{degree} of $u$, and is denoted by $d(u)$. Let $\Delta(G)=\Delta$ and $\overline{d}(G)=\overline{d}$ stand for the \emph{maximum} degree and the \emph{average} degree of the vertices of $G$. Clearly, we have $\Delta\ge\overline{d}.$

Let $k\ge 2$ be a fixed integer, and assume that the average degree of the vertices,
$$\overline{d}(G)=\sum_{u\in V}d(u)/n=cn.$$
We want to estimate the largest number $h(G)$ such that $G$ has $k$ vertices, $u_0, u_1,\ldots u_k$ such that $|N(u_0)\cup\ldots\cup N(u_{k-1})|\ge h(G).$
\smallskip

We start with two simple examples. Consider a random graph $\mathbf{G}(n,c)$ on $n$ vertices with edge probability $c\; (0<c<1)$. For a given set of $k$ vertices, $u_0,\ldots,u_{k-1},$ the probability that a fixed vertex different from them is not connected to any of them is $(1-c)^k$, Thus, the expected size of $$\bigcup_{i=0}^{k-1}N(u_i)\setminus\{u_0,\ldots,u_{k-1}\}=(n-k)(1-(1-c)^k)=(1-(1-c)^k+o(1))n.$$It easily follows by Chernoff's bound~\cite{AS, JLR} that, as $n\rightarrow\infty$, almost surely the size of the union of the neighborhoods of every $k$-tuple of vertices will be at most $(1-(1-c)^k+o(1))n.$ The same estimates hold for quasi-random graphs~\cite{CGW}.

Our second example is a graph $G'$ that consists of a complete graph with $\lceil\sqrt{c}n\rceil+1$ vertices and its remaining vertices are isolated. Clearly, the average degree of $G'$ satisfies $\overline{d}(G')\ge cn$, but the size of the union of no set of neighborhoods in $G'$ exceeds $\lceil\sqrt{c}n\rceil+1$.
\smallskip

At the problem session of the 60th birthday conference of Bruce Reed (Oxford, July 2022), Simon Griffiths conjectured that for every $c, \; 0<c<1,$ one of the above two graphs is asymptotically optimal, that is, the following theorem is true.

\begin{theorem}\label{main}
Let $k\ge 2$, and let $G=(V,E)$ be a graph on $n$ vertices with average degree $cn$, where $0\le c\le 1$.
Then there exist vertices $u_0,\dots,u_{k-1}\in V$ such that
\[
\bigl|N(u_0)\cup\cdots\cup N(u_{k-1})\bigr|
\ge
\min\bigl\{(1-(1-c))^k n,\,\sqrt c\,n\bigr\}.
\]
\end{theorem}

The problem was also featured as Questions 1 and 2 in the seminal paper of Jos\'e D. Alvarado, Leonardo Gon\c{c}alves de Oliveira, and Simon Griffiths~\cite{AlGG} (page 842).
\smallskip

In the next section, we prove the above result. Note that formally the theorem is true for $k=1$, too, and also for $c=0$ or $c=1$, but these cases are trivial.
\medskip

Before turning to the proof, we remark that an easy averaging argument shows that the theorem is asymptotically true for $k=2$, provided that $G$ is \emph{regular}, that is $d(u)-cn$ for every $u\in V$. Let $H$ denote the complement of $G$, and let $N_H(w)$ stand for the neighborhood of $w$ in $H$. Using the fact that $|N(u_1)\cup N(u_2)|\ge n-2-|N_H(u_1)\cap N_H(u_2)|,$ it is sufficient to show that there exist $u_1,u_2\in V$ such that
\begin{equation}\label{eq0}
|N_H(u_1)\cap N_H(u_2)|\le ((1-c)^2+o(1))n.
\end{equation}
To see this, estimate the sum of $|N_H(u_1)\cap N_H(u_2)|$ over all the $\binom{n}{2}$ unordered pairs of distinct vertices $\{u_1,u_2\}.$  We have
\[
\sum_{\{u_1,u_2\}}|N_H(u_1)\cap N_H(u_2)|=\sum_{v\in V}\binom{d_H(v)}{2}=n\binom{n-1-cn}{2}.
\]
Dividing the right-hand side by $\binom{n}{2},$ inequality (\ref{eq0}) follows.
\smallskip

After this initial success, we expected that, with a little extra effort, the argument could be extended to all degree sequences. We explored several plausible greedy strategies to find $k$ neighborhoods that collectively maximize vertex coverage; however, their outcomes proved resistant to analysis. Ultimately, we returned to the averaging approach, albeit with a small tweak: we started the procedure with a vertex of maximum degree. Miraculously, this slight modification made the underlying calculations tractable. The devil resides in the details of the calculus!

\section{Proof of Theorem~\ref{main}}

Choose a vertex $u_0\in V$ of \emph{maximum} degree $\Delta,$ and denote its neighborhood in $N(u_0)$ by $S$. Write
\[
d(u_0)=\Delta=sn,\qquad S:=N(u_0),\qquad T:=V\setminus S.
\]
Since $\Delta$ is at least the average degree $\overline{d}(G)$, we have $s\ge c$.

If $s\ge \sqrt c$, then already
\[
|N(u_0)|=\Delta=sn\ge \sqrt c\,n,
\]
and we are done.

Thus, we may assume throughout that
\[
c\le s<\sqrt c.
\]

Since every vertex of $S$ has degree at most $\Delta$, we have
\[
\sum_{u\in V(G)} |N(u)\cap S|
=\sum_{s\in S} d(s)
\le |S|\,\Delta
=\Delta^2
=s^2n^2.
\]
As the total degree sum is $cn^2$, it follows that
\begin{equation}\label{eq1}
\sum_{u\in V(G)} |N(u)\cap T|
=\sum_{u\in T} d(u)
=cn^2-\sum_{u\in V(G)} |N(u)\cap S|
\ge (c-s^2)n^2.
\end{equation}
\medskip

Choose $u_1,\dots,u_{k-1}\in V(G)$ independently and uniformly at random.
For any $u\in T$. The probability that $u$ lies in at least one of the neighborhoods
$N(u_1),\dots,N(u_{k-1})$ is
\[
1-\left(1-\frac{d(u)}{n}\right)^{k-1}.
\]
Notice that the function
\begin{equation}\label{eq2}
f(x):=1-(1-x)^{k-1}\qquad (0\le t\le 1).
\end{equation}
is nonnegative, concave on the interval $[0,1]$, and we have $f(0)=0, f(1)=1$.

Since $d(u)/n\le \Delta/n=s\le\sqrt{c}\le 1$, by concavity we obtain

\[
\frac{f\!\left(d(u)/n\right)}{d(u)/n}
\ge\frac{f(s)}{s}.
\]

Therefore, in view of (\ref{eq1}), the expected number of vertices of $T$ covered by
$N(u_1)\cup\cdots\cup N(u_{k-1})$ is at least
\[
\frac{f(s)}{s}\cdot \frac{1}{n}\sum_{u\in T} d(u)
\ge \frac{1-(1-s)^{k-1}}{s}\,(c-s^2)n.
\]
Hence, for a suitable choice of (not necessarily distinct) $u_1,\dots,u_{k-1}\in V$, we have
\[
\bigl|(N(u_1)\cup\cdots\cup N(u_{k-1}))\cap T\bigr|
\ge
\frac{1-(1-s)^{k-1}}{s}\,(c-s^2)n.
\]
Adding the neighborhood $S=N(u_0)$, we conclude that
\[
\bigl|N(u_0)\cup N(u_1)\cup\cdots\cup N(u_{k-1})\bigr|
\ge \left(s+\frac{1-(1-s)^{k-1}}{s}(c-s^2)\right)n.
\]
\medskip

Thus, it remains to prove that for $c\le s\le \sqrt c,$ the last expression is at least $$\min\{1-(1-c)^k,\sqrt c\}.$$
This is the content of the next lemma.

Let $k\ge 2$. To simplify the notation, set
\[
F(s):=s+\frac{1-(1-s)^{k-1}}{s}(c-s^2),\qquad  G(c):=1-(1-c)^k.
\]

\begin{lemma}\label{lemma2}
Let $k\ge 2$ be fixed. For every $\;0<c<1$ and every $s\in[c,\sqrt c]$, we have
\[
F(s)\ge \min\{G(c),\sqrt c\}.
\]
\end{lemma}

\begin{proof}
We split the proof into two cases.

\medskip
\noindent\textbf{Case 1: $G(c)\le \sqrt c$.}
\smallskip

It will be convenient to set a new variable
\[
x:=1-s,
\]
so that we have $x\in[1-\sqrt c,\,1-c]$ and
\[
F(s)=(1-x)x^{k-1}+c\,\frac{1-x^{k-1}}{1-x}=(1-x)x^{k-1}+c\sum_{i=0}^{k-2}x^i.
\]
Hence,
\begin{align*}
F(s)-G(c)
&=x^{k-1}-x^k+\sum_{i=0}^{k-2}x^i-(1-c)\sum_{i=0}^{k-2}x^i-1+(1-c)^k \\
&=(1-c)^k-(1-c)\sum_{i=0}^{k-2}x^i+\sum_{i=1}^{k-1}x^i-x^k \\
&=\bigl((1-c)^k-x^k\bigr)+\bigl(x-(1-c)\bigr)\sum_{i=0}^{k-2}x^i \\
&=\bigl((1-c)-x\bigr)\left(\sum_{i=0}^{k-1}(1-c)^{k-1-i}x^i-\sum_{i=0}^{k-2}x^i\right).
\end{align*}

We need to show that $F(s)-G(c)\ge 0$. Observe that $1-c-x\ge0.$ If $1-c=x,$ then we are done. Otherwise,
\[
\frac{F(s)-G(c)}{1-c-x}=\sum_{i=0}^{k-1}(1-c)^{k-1-i}x^i-\sum_{i=0}^{k-2}x^i.
\]
The right-hand side is a polynomial $P(x)$ of degree $k-1$ in $x$, with coefficients
\[
(1-c)^{k-1}-1,\ (1-c)^{k-2}-1,\ \dots,\ (1-c)-1,\ 1.
\]
Since $0<1-c<1$, this sequence has exactly one sign change. By Descartes' rule of signs, $P(x)$ has at most one positive root.
In view of the fact that $P(0)=(1-c)^{k-1}-1<0,$ it follows that if a positive root exists, then $P(x)<0$ before the root and $P(x)>0$ after that.

Evaluating $P$ at $x_0:=1-\sqrt c$ and taking into account our assumption $G(c)\le \sqrt c$,  we find that
\[
P(x_0)=\frac{F(\sqrt c)-G(c)}{1-c-x_0}=\sqrt{c}-G(c)\ge 0.
\]
Hence, $P(x)\ge 0$ for every $x\in[x_0,1-c]$, which implies that
$F(s)-G(c)\ge 0,$ as required.

\medskip
\noindent\textbf{Case 2: $G(c)\ge \sqrt c$.}
\smallskip

As before, we set $x:=1-s,$ so that $x\in[1-\sqrt{c},\,1-c]$. Now we need to show that $F(s)\ge\sqrt c$.


A direct calculation, similar to the one used in Case 1, gives that
\begin{equation}\label{eq3}
F(s)-\sqrt c=(\sqrt c-s)Q(x),
\end{equation}
where
\[
Q(x):=\sqrt{c}\sum_{i=0}^{k-2}x^i-x^{k-1}
\]
is a polynomial of degree $k-1$ is $x$.

Since $\sqrt{c}-s\ge 0$, it remains to show that $Q(x)\ge 0$ on the interval $[1-\sqrt{c},1-c]$.

The coefficients of $Q(x)$ are
\[
\sqrt{c},\ \sqrt{c},\ \dots,\ \sqrt{c},\ -1,
\]
so there is exactly one sign change.
By Descartes' rule of signs, $Q(x)$ has at most one positive root.
We have
$Q(0)=\sqrt{c}>0,$
and the leading coefficient is $-1$. Thus, if a positive root exists, then $Q(x)>0$ before the root and $Q(x)<0$ after that.
Now let $x_0:=1-c.$
We compute
\begin{align*}
Q(x_0)
&=\sqrt{c}\sum_{i=0}^{k-2}(1-c)^i-(1-c)^{k-1}=\sqrt{c}\cdot \frac{1-(1-c)^{k-1}}{c}-(1-c)^{k-1} \\
&=\frac{1-(1-c)^{k-1}}{\sqrt{c}}-(1-c)^{k-1}=\frac{1-(1-c)^{k-1}(1+\sqrt{c})}{\sqrt{c}}.
\end{align*}
On the other hand,
\begin{align*}
G(c)\ge \sqrt{c}
&\iff 1-(1-c)^k\ge \sqrt{c} \\
&\iff (1-c)^{k-1}(1-\sqrt{c})(1+\sqrt{c})\le 1-\sqrt{c} \\
&\iff (1-c)^{k-1}(1+\sqrt{c})\le 1,
\end{align*}
Thus, our assumption $G(c)\ge \sqrt{c}$ is equivalent to the statement that
$Q(x_0)\ge 0.$
Therefore, any positive root of $Q$ is at least $x_0$, so $Q(x)\ge 0$ for all $x\in[0,x_0]$.
In particular, we have $Q(x)\ge 0$ for all $x\in[1-\sqrt{c},1-c]$. In view of (\ref{eq3}), this implies that
\[
F(s)-\sqrt{c}\ge 0.
\]
as required.
\medskip

Combining the above two cases yields
\[
F(s)\ge \min\{G(c),\sqrt c\},
\]
which completes the proof of Lemma~\ref{lemma2} and, hence, the theorem.
\end{proof}

\bigskip

\noindent
{\bf \large Acknowledgements.}
J\'anos Pach was supported by ERC Advanced Grant `GeoScape' No.\ 882971 and by the National Research, Development and Innovation Office, NKFIH, K-131529. He also acknowledges the hospitality of Academia Sinica, Taipei and the Matrix Research Institute, Creswick, where this note was written. He is grateful to P\'eter Frankl, Simon Griffiths, Bruce Reed, and G\'abor Tardos for their valuable suggestions.

\end{document}